\documentclass[12pt, reqno]{amsart}
\usepackage{amsmath,amsthm,amssymb,amsfonts}

\theoremstyle{plain}
\newtheorem{thm}{Theorem}[section]
\newtheorem{prop}[thm]{Proposition}
\newtheorem{lem}[thm]{Lemma}

\newtheorem{cor}[thm]{Corollary}
\newtheorem{conj}[thm]{Conjecture}

\theoremstyle{definition}

\newtheorem{rem}[thm]{Remark}
\newtheorem{defn}[thm]{Definition}
\newtheorem{eg}[thm]{Example}
\newtheorem{subtitle}[thm]{}
\newtheorem{ex}{Exercise}[section]
\numberwithin{equation}{section}

\def\ms{\medskip}

\def\ni{\noindent}
\def\ti{\tilde}

\def\R{\mathbb{R} }
\def\C{\mathbb{C}}

\def\Z{\mathbb{Z}}

\newcommand{\beg}{\begin{eg}}
\newcommand{\eeg}{\end{eg}}
\newcommand{\bthm}{\begin{thm}}
\newcommand{\ethm}{\end{thm}}
\newcommand{\bprop}{\begin{prop}}
\newcommand{\eprop}{\end{prop}}
\newcommand{\bcor}{\begin{cor}}
\newcommand{\ecor}{\end{cor}}
\newcommand{\blem}{\begin{lem}}
\newcommand{\elem}{\end{lem}}
\newcommand{\bca}{\begin{cases}}
\newcommand{\eca}{\end{cases}}
\newcommand{\brem}{\begin{rem}}
\newcommand{\erem}{\end{rem}}
\newcommand{\bconj}{\begin{conj}}
\newcommand{\econj}{\end{conj}}
\newcommand{\bpm}{\begin{pmatrix}}
\newcommand{\epm}{\end{pmatrix}}
\newcommand{\bbm}{\begin{bmatrix}}
\newcommand{\ebm}{\end{bmatrix}}
\newcommand{\bvm}{\begin{vmatrix}}
\newcommand{\evm}{\end{vmatrix}}
\newcommand{\bdefn}{\begin{defn}}
\newcommand{\edefn}{\end{defn}}
\newcommand{\bsub}{\begin{subtitle}}
\newcommand{\esub}{\end{subtitle}}
\newcommand{\bex}{\begin{ex}}
\newcommand{\eex}{\end{ex}}
\newcommand{\ben}{\begin{enumerate}}
\newcommand{\een}{\end{enumerate}}

\begin{document}

\title[Geometric automorphism groups of  symplectic
$4$-manifolds]
{Geometric automorphism groups of  symplectic
$4$-manifolds}
\author{Bo Dai}
\address{LMAM, School of Mathematical Sciences \\ Peking
University \\ Beijing 100871, P. R. China}
\email{daibo@math.pku.edu.cn}
\author{Chung-I Ho}
\address{National Center for Theoretical Sciences\\
Math. Division\\ Hsinchu, Taiwan}
\email{ciho@math.cts.nthu.edu.tw}
\author{Tian-Jun Li}
\address{Department of Mathematics \\
University of Minnesota\\ Minneapolis, MN 55455}
\email{tjli@math.umn.edu}

\ms

\begin{abstract} Let $M$ be a closed, oriented, smooth $4-$manifold with intersection form
$\Gamma$, $A(\Gamma)$  the automorphism group of $\Gamma$ and
$D(M)$  the subgroup induced by orientation-preserving
diffeomorphisms of $M$.  In this note  we study the question when
 $D(M)$ is of infinite index in $A(\Gamma)$ for a  symplectic
4--manifold.
\end{abstract}

\maketitle
\tableofcontents

\section{Introduction}
For a given unimodular symmetric bilinear form $\Gamma$, let
$A(\Gamma)$ be its automorphism group. Given a closed, oriented,
topological $4-$manifold $M$, let $\Lambda_M$ be  the free abelian group obtained from $H^2(M;\mathbb Z)$ by modulo torsion, and 
$\Gamma_M$  the associated
unimodular symmetric bilinear form, namely, the intersection form
on $\Lambda_M$. By a celebrated result of
Freedman, any  unimodular symmetric bilinear form   is realized as
the intersection form of an oriented, simply connected topological
$4-$manifold.   Moreover, for such a topological manifold $M$, the
natural map from the group of orientation-preserving
homeomorphisms to $A(\Gamma_M)$ is surjective.

For a smooth, closed, oriented $4-$manifold $M$ with intersection
form $\Gamma$, there is a natural map from the group of
orientation-preserving diffeomorphisms Diff$^+(M)$ to the
automorphism group of  $\Gamma$, $A(\Gamma)$. Let $D(M)$ be the
image of this natural map. In other words, an automorphism is in
$D(M)$ if it is realized by an orientation-preserving
diffeomorphism.  $D(M)$ is  called the geometric automorphism
group. The group $D(M)$, both as an abstract group and as a
subgroup of $A(\Gamma)$, is a powerful smooth invariant, which is
nonetheless  hard to compute in general.

Wall initiated  the comparison of $D(M)$ and $A(\Gamma)$ in a
series of papers \cite{wall-uni}, \cite{wall-ii}, \cite{W64}. In
particular, he proved in \cite{W64} the following beautiful
result: for any simply connected smooth manifold with $\Gamma$
strongly indefinite or of rank at most $10$, if there is  an
$S^2\times S^2$ summand in its connected sum decomposition, then
$D(M)=A(\Gamma)$. For K\"ahler surfaces, especially elliptic
surfaces, rational surfaces, ruled surfaces, we have a rather good
understanding of $D(M)$ due to Friedman, Morgan, Donaldson, L\"onne
\cite{FM88}, \cite{FM97}, \cite{Do90},  \cite{lo} (see also \cite{LL03},
\cite{LL05}).

In this note we will focus on the  question when $D(M)$ is of
infinite index in $A(\Gamma_M)$ if $A(\Gamma_M)$ is an infinite
group. We first observe in Theorem  \ref{a} that $A(\Gamma)$ is
infinite  if $\Gamma$ is indefinite of rank at least 3. Moreover,
we offer a simple criterion for a subgroup to have infinite index.
We apply this criterion to symplectic manifolds and obtain an
almost complete answer.

To state our result, let us first recall some definitions.
For a smooth $4-$manifold $M$ with
a symplectic form $\omega$, let $K_{\omega}$ denote the symplectic
canonical class. A symplectic $4-$manifold   is said to be minimal
if it does not contain any embedded symplectic sphere with
self-intersection $-1$.  A general symplectic $4-$manifold $(M,
\omega)$ can be symplectically blown down to a minimal one, which
is called a minimal model.

The Kodaira dimension of a
symplectic $4-$manifold $(M, \omega)$ is defined below.

\begin{defn}
If $(M, \omega)$ is minimal,  the Kodaira dimension of $(M,\omega)$
 is defined in the following way:

\[ \kappa(M,\omega)=\left\{ \begin{array}{ll}
-\infty &\hbox{if $K_{\omega}\cdot [\omega]<0$ or $K_{\omega}\cdot K_{\omega}<0$},\\
0&\hbox{ if $K_{\omega}\cdot [\omega]=0$ and $K_{\omega}\cdot K_{\omega}=0$},\\
1&\hbox{ if $K_{\omega}\cdot [\omega]> 0$ and $K_{\omega}\cdot K_{\omega}=0$},\\
2&\hbox{ if $K_{\omega}\cdot [\omega]>0$ and $K_{\omega}\cdot K_{\omega}>0$}.\\
\end{array}
\right. \]

 For a general $(M, \omega)$, $\kappa(M,\omega)$ is
defined to be that of any of its minimal model.
\end{defn}

It is shown in \cite{L06} that $\kappa(M,\omega)$ is well-defined
and agrees with the holomorphic Kodaira dimension if $(M, \omega)$
is K\"ahler. Moreover, it turns out that $\kappa(M,\omega)$ only
depends on $M$ so we will denote it by $\kappa(M)$.

\begin{thm}\label{main}
Suppose $M$ has symplectic structures and $A(\Gamma_M)$ is
infinite. Then $D(M)$ is of infinite index if
\begin{itemize}

\item $\kappa(M) =-\infty$, and $M=\mathbb C \mathbb P^2\#
n\overline{\mathbb C \mathbb P}^2$ with $  n\geq 10$ or
$(\Sigma\times S^2)\# n\overline{\mathbb C \mathbb P}^2$ with
$n\geq 1$, where $\Sigma$ is a closed Riemann surface of positive
genus.

\item  $\kappa(M)=0$ and $\Gamma_M$ is odd.

\item  $\kappa(M)\geq 1$.

\end{itemize}
\end{thm}

This result follows from Propositions \ref{-infty}, \ref{ab},  \ref{odd}.

$M$ is called a symplectic Calabi-Yau surface if there is a
symplectic form $\omega$ on $M$ such that $K_{\omega}$ vanishes in
the real cohomology. The third author showed in \cite{L06} that $M$ 
is a symplectic Calabi-Yau surface exactly when $\kappa(M)=0$ and 
$\Gamma_M$ is even. With  this understood,  Theorem \ref{main} can 
be restated as: When $M$ is
symplectic and $A(\Gamma_M)$ is infinite, $D(M)$ is of finite
index only when $M$ is a  symplectic Calabi-Yau surface, or
$\mathbb C \mathbb P^2\# n\overline{\mathbb C \mathbb P}^2$ with $
2\leq n\leq 9$.

We define K\"ahler Calabi-Yau surfaces in the same way. There are
three  K\"ahler Calabi-Yau surfaces with infinite $A(\Gamma)$: K3
surface,  Enriques surface, $T^4$. All of them have finite index
geometric automorphism group. The only known non-K\"ahler
Calabi-Yau surfaces with infinite $A(\Gamma)$  are the so-called
Kodaira-Thurston manifolds.
 We will show in the last section that they have infinite
index geometric automorphism group. Thus
we further make the following conjecture.

\begin{conj} \label{kahler} Suppose $M$ has symplectic structures and $A(\Gamma_M)$ is infinite.
Then $D(M)$ is of finite index if and only if  $M$ is

 \begin{itemize}
\item  a K\"ahler Calabi-Yau surface, or

\item $\mathbb C \mathbb P^2\# n\overline{\mathbb C \mathbb P}^2$ with $ 2\leq n\leq 9$.

\end{itemize}
\end{conj}


\section{Infinite $A(\Gamma)$}
\subsection{Quadratic forms}
Let $\Lambda$ be a  finitely generated free
abelian group, and  $\Gamma :\Lambda\times \Lambda\rightarrow \mathbb{Z}$  a
unimodular symmetric bilinear form on $\Lambda$. Sometimes we abbreviate $\Gamma(x, y)$ as $x\cdot y$.
It induces a quadratic form
$Q:\Lambda\rightarrow \mathbb{Z}$ as $Q(x)=\Gamma(x, x)$. $Q(x)$
is called the norm of $x$. $\Gamma$ is of even type if $Q(x)$ is
even for any vector $x \in \Lambda$. Otherwise, it is called of
odd type. The rank $r(\Gamma)$ of $\Gamma$ is the rank of
$\Lambda$. Let $b^+,b^-$ be the number of 1, -1 respectively, on a
diagonal matrix over $\mathbb{R}$ representing $\Gamma$. The
signature $\sigma(\Gamma)$ of $\Gamma$ is the difference
$b^+-b^-$.

\begin{defn} $\Gamma$
is called {\it definite, nearly definite\/} or {\it strongly
indefinite\/} if min$\{b^+,b^-\}=0,1$ or $\geq2$ respectively.
\end{defn}

The following classification is well known, see e.g. \cite{wall-uni}.
\begin{thm}\label{classification}
The classification of indefinite unimodular symmetric forms is given by their rank, signature and type.
\end{thm}

Let $U,E$ be respectively the hyperbolic lattice and the (positive definite) $E_8$
lattice. The list of indefinite unimodular symmetric forms are
$$m\langle 1 \rangle \oplus n\langle -1\rangle, \quad pU\oplus qE, \
 \text{where} \ m,n,p\in\mathbb{N},q\in \mathbb{Z}$$

\subsection{Infinite $A(\Gamma)$ and criterion for subgroups of infinite index}
Let $A(\Gamma)$ be the automorphism group of $\Gamma$.
In this subsection we establish  the following  criterion for subgroups of $A(\Gamma)$ to have infinite index.

\begin{thm}\label{a} 
Let $A(\Gamma)$ be the automorphism group of $\Gamma$.
\begin{itemize}
\item $A(\Gamma)$ is finite if and only if it is definite or indefinite
of rank $2$.

\item Suppose $A(\Gamma)$ is infinite, ie. $\Gamma$ is indefinite of rank $\geq 3$. If there are
finitely many nonzero characteristic classes invariant under a
subgroup $D$, then $D$ is of infinite index.

\end{itemize}
\end{thm}

\subsubsection{Infinite transitive actions for strongly indefinite $\Gamma$}
A vector $x\in \Lambda$ is called primitive if it cannot be
divided by any integer except $\pm 1$. A vector $x$ is called
characteristic if for all vectors $y$ in $\Lambda$, $x\cdot
y\equiv y\cdot y\  (\text{mod} \ 2)$. Otherwise, it is called
ordinary.
It is clear that $A(\Gamma)$ preserves the norm and type of each vector.
When $\Gamma$ is strongly indefinite, $A(\Gamma)$ often acts
transitively with infinite orbits.

\begin{prop}\label{transitive}
Assume $\Gamma$ is strongly indefinite.
\begin{enumerate}
\item $A(\Gamma)$ acts transitively on primitive vectors of given
norm and type.

\item For any $k\in\mathbb{Z}$, there are infinitely many nonzero characteristic classes of norm
$\sigma(\Gamma)+8k$.

\end{enumerate}
\end{prop}

\begin{proof}
\begin{enumerate}
\item
See Theorem 6 in \cite{wall-uni}.

\item  It is enough to consider the case $r(\Gamma)=4$. General case then follows by
extension.

If $\Gamma=2U$ and $x_0, y_0, x_1, y_1$ is a basis of $\Lambda$
such that  $x_0\cdot y_0=x_1\cdot y_1=1$, we know that any
characteristic class is of the form $x=2ax_0+2by_0+2cx_1+2dy_1$
for some $a,b,c,d\in \mathbb{Z}$. We have $Q(x)=8(ab+cd)$ and it
is clear that there are infinitely many quadruple $(a,b,c,d)$
satisfying $ab+cd=k$. For instance, $ak+(1-a)k=k$ for any $a$.

If $\Gamma$ is even, then it is of the form $2U\oplus L$. Any
characteristic vector $c$ for $2U$ gives rise to a characteristic
vector $(c, 0)$ for $\Gamma$ of the same norm.

If $\Gamma=2<1>+2<-1>$ and $\Lambda $ has a basis $p_1, p_2,
q_1,q_2$ with  $p_1^2=p_2^2=-q_1^2=-q_2^2=1$, any characteristic
class is of the form $x=ap_1+bp_1+cq_1+dq_2$
for some odd integers $a,b,c,d$. If $k=2^tr$ and $r$ is odd, we
can set $a=2^{t+1}+r$ and $c=2^{t+1}-r$.
Then $Q(ax_1+cy_1)=8k$. So $Q(ap_1+bp_1+cq_1+bq_2)=8k$ for any
$b$.

If $\Gamma$ is odd, then it is of the form $(2<1>+2<-1>)\oplus L$.
Any characteristic vector $c$ for $2<1>+2<-1>$ gives rise to a
characteristic vector $(c, (1))$ for $\Gamma$ whose norm differs
from that of $c$ by a fixed constant.

\end{enumerate}
\end{proof}

In some cases, higher dimensional subspaces also have such transitive property.
A subgroup $W\subset \Lambda$ is called full if $(W\otimes \mathbb{Q})\cap \Lambda=W$.
We define
$$\mathcal{U}(\Gamma)=\{W\subset \Lambda | W:\text{ full,
isotropic, of dimension 2}\}$$ 

\begin{lem}\label{transitive'} $A(2U)$ acts transitively on $\mathcal{U}(2U)$ and $|\mathcal{U}(2U)|=\infty$.
\end{lem}
\begin{proof}
Using the notations in the proof of Proposition~\ref{transitive},
if $Y\in\mathcal{U}(2U)$ and $x, y\in Y$ are linearly independent primitive vectors,
Proposition~\ref{transitive} (1) implies that $\alpha x=x_0$ for
some $\alpha\in A(2U)$. So $\alpha y=ax_0+bx_1$ or $ax_0+by_1$ for
some $a, b$. Hence $\alpha (Y)=<x_0,x_1>$ or $<x_0, y_1>$ and
$\mathcal{U}(2U)$ is transitive.
Moreover, we can show that
$$\mathcal{U}(2U)=\{<ax_0+bx_1, by_0-ay_1>,<ax_0+by_1, by_0-ax_1>|gcd(a,b)=1\}.$$
So $|\mathcal{U}(2U)|=\infty$.
\end{proof}

\subsubsection{Infinite orbits  for nearly definite $\Gamma$}

Now consider the case that $\Gamma$ is nearly definite, i.e.
$b^+=1$ or $b^-=1$.
In this case, we cannot always establish transitivity of actions. Instead we
show the infiniteness of orbits.

\begin{lem}\label{key-lemma}
Let $\Gamma$ be nearly definite of rank  at least $3$. For any
nonzero $x\in \Lambda$,  the orbit of $x$ under $A(\Gamma)$ is
infinite.
\end{lem}

\begin{proof}
We only consider nearly positive definite case. First consider odd
type case. Let $H_1, \cdots, H_n, F,\ n\geq 2$, be an orthogonal
basis of $\Lambda$ such that $H_i^2=-F^2=1$. Let $x=a_1H_1+\cdots
+a_nH_n+bF$. Without loss of generality, we may assume
$|a_1|\geq|a_2|\geq\cdots\geq |a_n|$. Consider the reflection
$R_\gamma$ of $\gamma=\varepsilon_1 H_1+\varepsilon_2 H_2+F$ where
$\varepsilon_i=\pm 1$. It is easy to see that
\begin{eqnarray*}
R_\gamma(F) &=& 2\varepsilon_1H_1+2\varepsilon_2H_2+3F\\
 R_\gamma(H_1) &=& -H_1-2\varepsilon_1\varepsilon_2H_2-2\varepsilon_1F\\
 R_\gamma(H_2)&=& -2\varepsilon_1\varepsilon_2H_1-H_2-2\varepsilon_2F
 \end{eqnarray*}
So
\begin{eqnarray*}
R_\gamma(x) &=
(-a_1-2\varepsilon_1\varepsilon_2a_2+2\varepsilon_1b)H_1
+(-2\varepsilon_1\varepsilon_2a_1-a_2+2\varepsilon_2b)H_2  \\
\quad &+ a_3H_3+\cdots +a_nH_n
+(-2a_1\varepsilon_1-2a_2\varepsilon_2+3b)F
\end{eqnarray*}
Choosing $\varepsilon_1, \varepsilon_2$ appropriately, the
coefficient of $F$ in $R_\gamma(x)$ is monotone.

Now consider even type case. $\Gamma$ is equivalent to $U+lE$ with
$l>0$. Let $x,y$ be a standard basis of $U$, i.e. $x^2=y^2=0$,
$x\cdot y=1$. Then any (nonzero) class in $\Gamma$ has a unique
decomposition $\eta+ax+by$, where $\eta\in lE$. There are three cases.
\begin{enumerate}
\item $\eta\not=0$, $ax+by\not=0$. Without loss of generality,
assume $b\not=0$. Since $lE$ has a basis such that each vector has
square $2$, there exists an $\omega\in lE$ such that $\omega^2=2$,
and $\omega\cdot\eta\not=0$. For any $k\in \Z$, $(\omega+kx)^2=2$.
Consider \begin{eqnarray*}
R_{\omega+kx}(\eta&+&ax+by)=\eta+ax+by-(\omega\cdot\eta
+kb)(\omega+kx) \\
&=&\eta -(\omega\cdot\eta +kb)\omega +(a-(\omega\cdot\eta +kb)k)x
+by. \end{eqnarray*}
 We can choose $k\in \Z$ such that $\eta
-(\omega\cdot\eta +kb)\omega\not=0$, and the coefficient of $x$ is
monotone (decreases if $b>0$, and increases if $b<0$). Repeating
this process, we see that the orbit is infinite.

\item $0\not=\eta\in lE$, $ax+by=0$. Choose $\omega\in lE$ such
that $\omega^2=2$. Consider
$$R_{\omega+y}(\eta) =\eta -(\omega\cdot\eta) (\omega+y) =\eta
-(\omega\cdot\eta) \omega -(\omega\cdot\eta)y. $$ By properties of
$E$, we can choose $\omega$ such that $\omega\cdot\eta\not=0$, and
$\eta -(\omega\cdot\eta) \omega\not=0$. Then we are back to case
(1).

\item $\eta=0$, $ax+by\not=0$. We may assume $b\not=0$. Choose
$\omega\in lE$ such that $\omega^2=2$, and $k\not=0$. Consider
$$ R_{\omega+kx}(ax+by) =ax+by-kb(\omega+kx) =-(kb)\omega
+(a-k^2b)x +by. $$ Then we are back to case (1) again.
\end{enumerate}
\end{proof}

\subsubsection{Proof of Theorem  \ref{a}}

\begin{proof} Let us first show that $A(\Gamma)$ is finite if and only if
$\Gamma$ is definite or indefinite of rank $2$. The if part is
known, namely, if $\Gamma$ is definite or indefinite of rank $2$,
then $A(\Gamma)$ is finite. See the remarks after conclusion of
\cite{wall-uni}.

For the only if part, when $\Gamma$ is strongly
indefinite, it follows from Proposition  \ref{transitive} (1) and (2).

 When $\Gamma$
is nearly definite of rank $\geq 3$, it follows from
Lemma~\ref{key-lemma}.

Notice that exactly the same argument proves the statement in the second bullet.
\end{proof}


\section{$D(M)$ with infinite index}
Let $M$ be a closed, oriented, smooth $4-$manifold. A symplectic
form on $M$ is a closed  $2-$form $\omega$ on $M$ such that
$\omega\wedge \omega$ is a volume form inducing the given
orientation of $M$. Given $\omega$, it comes with a
contractible set of  almost complex structures tamed by $\omega$.
Suppose an almost complex structure $J$ is from this contractible
set. The canonical class of $\omega$ is then defined to be
$-c_1(M, J)$, and denoted by $K_{\omega}$. We call $K_{\omega}\in
H^2(M;\mathbb Z)$ a symplectic canonical class of $M$. It is a
characteristic class in $\Gamma_M$ with norm
$2\chi(M)+3\sigma(\Gamma_M)$, where $\chi(M)$ is the Euler number of $M$.

Let
$$\mathcal K_M=\{K_{\omega}|\omega \hbox{ a symplectic form on $M$}\}$$
 be the set of symplectic canonical classes
of $M$.  Clearly, $\mathcal K_M$ is nonempty if and only if $M$
has symplectic structures. Let $\overline {\mathcal K}_M$ be the
image of $\mathcal K_M$ in $\Gamma_M$.

\begin{thm}\label{canonical classes}
$\overline{\mathcal K}_M$ has the following properties.

\begin{itemize}
\item $\overline{\mathcal K}_M$ is preserved by $D(M)$.

\item Suppose $M$  has symplectic structures and  $\kappa(M)\geq
0$. Then $\overline{\mathcal K}_M$ is a finite set.

\item  $\overline {\mathcal K}_M$ contains $0$ if and only if $M$
is a symplectic Calabi-Yau surface.
\end{itemize}
\end{thm}

\begin{proof} The first statement follows from the following simple observation:
For any symplectic form $\omega$ and orientation preserving
diffeomorphism $\phi$, $\phi^*\omega$ is still a symplectic form,
and $K_{\phi^*\omega}=\phi^*K_{\omega}$.

The second statement in the case $b^+\geq 2$ follows from Taubes's
fundamental results in  \cite{T94} and \cite{T96},    and in  the
case  $b^+=1$ it is established  in  \cite{LLiu}.

The last statement is noted in \cite{L06}.
\end{proof}

We will denote $\Gamma_M$ as $\Gamma$ in the following.

\subsection{$\kappa=-\infty$}
There is a classification of $\kappa=-\infty$ manifolds: $M$ is either rational or ruled (\cite{Liu}, \cite{OO}).

 For a rational $4$-manifold $M=\C \mathbb P^2 \# n\overline{\C
\mathbb P^2}$ with $n\leq 9$ or $S^2\times S^2$, a classical result of Wall \cite{W64} says that $D(M)$
coincides with $A(\Gamma)$.

For $M=\C \mathbb P^2 \# n \overline{\C
\mathbb P}^2$ with $n\geq 10$, Friedman and Morgan \cite{FM88} showed that
$D(M)$ is a  subgroup of $A(\Gamma)$ with infinite index, and
characterized it in terms of super $P-$cells. Another proof of these
results appeared in \cite{LL05} by presenting an explicit and finite
generating set of $D(M)$.

The case  of irrational ruled $4$-manifolds has been studied in
\cite{FM97} and \cite{LL03}. Let $\Sigma$ be a closed Riemann
surface of positive genus, and $\Sigma\ti\times S^2$ be the
nontrivial $S^2$-bundle over $\Sigma$. Then any minimal irrational
ruled manifold $M$ is diffeomorphic to $\Sigma\times S^2$ or
$\Sigma\ti\times S^2$ for some  $\Sigma$. For such manifolds, it is known that $A(\Gamma)\cong \Z_2
\oplus\Z_2$, and $D(M)\cong \Z_2$ \cite{LL97}. So $D(M)$ is a proper
subgroup of $A(\Gamma)$, and both are finite groups.

Any non-minimal
irrational ruled $4$-manifold is diffeomorphic to $(\Sigma\times
S^2)\# n\overline{\C \mathbb P^2}$ with $n\geq 1$. There is a unique
spherical class $f$ (up to sign) of square zero, namely the class
represented by the $S^2$ factor in the $\Sigma\times S^2$ summand.
In the case, Friedman and Morgan
proved that an automorphism $\tau\in A(\Gamma)$ is in $D(M)$ if and
only if $\tau(f)=\pm f$. By presenting an explicit and finite
generating set of $D(M)$, it was  proved in \cite{LL05} that for
$M=(\Sigma\times S^2)\# n\overline{\C \mathbb P^2}$ with $n\geq 1$,
$D(M)$ is a subgroup of $A(\Gamma)$ with infinite index. Alternatively, this also 
follows from Lemma \ref{key-lemma}.

We summarize the discussion in the following proposition.

\begin{prop}\label{-infty}

Suppose $\kappa(M)=-\infty$. Then

\begin{itemize}
 \item $D(M)=A(\Gamma)$ if $M=S^2\times S^2$ or $M=\C \mathbb P^2 \# n\overline{\C
\mathbb P^2}$ for $n\leq 9$.

\item $A(\Gamma)$ is finite and $D(M)$ is a subgroup of  index $2$  if $M$ is an $S^2-$bundle over
a positive genus surface.

\item $D(M)$ is of infinite index in all other cases.

\end{itemize}

\end{prop}

\subsection{$\kappa=0$ with $\Gamma$ odd, and $\kappa\geq 1$}

\subsubsection{$\kappa=0$ with $\Gamma$ odd}

\begin{prop}\label{odd}
When $\kappa(M)=0$ and $\Gamma$ is odd,
$A(\Gamma)$ is infinite and $D(M)$ is of infinite index in $A(\Gamma)$.
\end{prop}
\begin{proof}
When $\kappa(M)=0$ and $\Gamma$ is odd,  $M$ is non-minimal. Let
$M'$ be a minimal model. Then $M'$ is  a symplectic Calabi-Yau
surface.  From the table in the next section,  $b^-(M')\geq 1$.
Thus $b^-(M)\geq 2$. Since $M$ admits  a symplectic structure we
have $b^+(M)\geq 1$. Thus by  the first statement of Theorem
\ref{a}, $A(\Gamma)$ is infinite.

Since $M$ is non-minimal, the set  $\overline {\mathcal K}_M$ is
finite, consists of nonzero classes and is invariant under $D(M)$
by Theorem \ref{canonical classes}. Now apply Theorem \ref{a}.

\end{proof}

\subsubsection{$\kappa\geq 1$}

\begin{prop} \label{ab}

If $\kappa(M)\geq 1$ and  $A(\Gamma)$ is infinite,
then $D(M)$ is of infinite index in $A(\Gamma)$.
\end{prop}

\begin{proof} Since it is assumed that $A(\Gamma)$ is infinite, the conclusion
follows directly from Theorem~\ref{canonical classes} and the
second statement of  Theorem \ref{a}.
\end{proof}

\subsubsection{Proof of  Theorem \ref{main} }
\begin{proof} It follows from Propositions \ref{-infty}, \ref{ab}, \ref{odd}.
\end{proof}

\section{Symplectic Calabi-Yau surfaces}

In this section we will focus on symplectic CY surfaces.
Specifically we will provide evidences for Conjecture
\ref{kahler} by showing that

1.  for any  K\"ahler CY surface, $D$ is of finite index.

2. for any known non-K\"ahler CY surface, if $A(\Gamma)$ is
infinite, then $D$ is of infinite index.

\subsection{Homological classification}
A symplectic Calabi-Yau surface is a minimal manifold with $\kappa=0$.

There is a homological classification of symplectic CY surfaces in \cite{L4} and \cite{B}.

\begin{thm}
A symplectic CY surface is a $\Z-$homology K3 surface, a $\Z-$homology
Enriques surface or a $\mathbb Q-$homology $T^2-$bundle over $T^2$.
\end{thm}
The following table list possible homological invariants of
symplectic CY surfaces \cite{L4}:

$$\begin{tabular}{|c|c|c|c|c|c|}
\hline
$b_1$ & $b_2$ & $b^+$ & $\chi$ & $\sigma$ & known manifolds\\
\hline
0 & 22 & 3 & 24 & -16 & K3\\
\hline
0 & 10 & 1 & 12 & -8 & Enriques surface\\
\hline
4 & 6 & 3 & 0 & 0 & 4-torus\\
\hline
3 & 4 & 2 & 0 & 0 & $T^2-$bundles over $T^2$\\
\hline
2 & 2 & 1 & 0 & 0 & $T^2-$bundles over $T^2$\\
\hline
\end{tabular}
$$

It is also speculated that in fact a symplectic CY surface is
actually the K3 surface, Enrique surface or a $T^2-$bundle over
$T^2$.

\subsection{K\"ahler CY surfaces}

\begin{prop}\label{kahler'}
$D(M)$ is of  index at most $4$ if $M$ is a K\"ahler CY surface.

\end{prop}
\begin{proof} According to the Kodaira classification, a K\"ahler CY surface is
either a hyperelliptic surface, the Enrique surface, the K3 surface, or $T^4$.

For hyperelliptic surfaces, $\Gamma=U$, so $A(\Gamma)$ is the order $4$ group $\Z_2\oplus\Z_2$.
For the K3 surface, it was shown by Donaldson in \cite{Do90} that  $D$ is of index 2.
 For the
Enriques surface it was shown by L\"{o}nne   in \cite{lo} that $D=A$. It remains to deal with $T^4$. Our claim is that  that $[A(3U):D(T^4)]=4$.

Use the standard model $\R^4/\Z^4$ for $T^4$, with coordinates
$t_1,\dots,t_4$. Then $H^1$ is generated by $\{ dt_1,\dots dt_4
\}$, and $H^2$ is generated by $\{ dt_i \wedge dt_j,\  i<j \}$.
Let
\begin{eqnarray*}
x_1&=dt_1\wedge dt_2, \quad y_1=dt_3\wedge dt_4, \\
x_2&=dt_1\wedge dt_3, \quad y_2=dt_4\wedge dt_2, \\
x_3&=dt_1\wedge dt_4, \quad y_3=dt_2\wedge dt_3.
\end{eqnarray*}
 Then the nonzero relations are $x_i\cdot y_i
=y_i \cdot x_i=1$, and $\Gamma =3U$.

In the following, we define certain automorphisms of $\Gamma$ by
listing the non-invariant terms
$$\begin{array}{rl}
n_i: & x_i\mapsto -x_i, \ y_i \mapsto -y_i\\
s_i: & x_i\mapsto y_i, \ y_i \mapsto x_i\\
p_{ij}(=p_{ji}): & x_i\mapsto x_j, \ y_i \mapsto y_j,\ x_j\mapsto x_i, \ y_j \mapsto y_i\\
\alpha_{ij}: & x_i\mapsto x_i+x_j, \ y_j \mapsto y_j-y_i
\end{array}$$

Wall showed in \cite{wall-ii} that $A(3U)$ is generated by $n_i, s_i, p_{ij}$ and $\alpha_{ij}$.

If $A=(a_{ij})\in SL(4,\Z)$ which induces automorphism
$A(dt_j)=\sum a_{ij}dt_i$ of $H^1$, then
$$ \Lambda^2 A(dt_i \wedge dt_j) =\sum_{k,l} a_{ki}a_{lj} dt_k \wedge
dt_l =\sum_{k<l} (a_{ki}a_{lj}-a_{li}a_{kj}) dt_k \wedge dt_l. $$
Let $C=\Lambda^2 A=(p_{kl,ij}) \in A(3U)$. Then we have relations
$$P_{kl,ij}: p_{kl,ij}=a_{ki}a_{lj}-a_{kj}a_{li}.$$
 Note that
$p_{kl,ij}=-p_{lk,ij}=-p_{kl,ji}$.
If we choose
$$A=
\begin{pmatrix}
-1 & 0 & 0 & 0\\
0 & 1 & 0 & 0\\
0 & 0 & 1 & 0\\
0 & 0 & 0 & -1\\
\end{pmatrix},
\begin{pmatrix}
0 & 0 & 0 & -1\\
0 & 0 & 1 & 0\\
0 & -1 & 0 & 0\\
1 & 0 & 0 & 0\\
\end{pmatrix},
\begin{pmatrix}
1 & 0 & 0 & 0\\
0 & 0 & 1 & 0\\
0 & -1 & 0 & 0\\
0 & 0 & 0 & 1\\
\end{pmatrix},
\begin{pmatrix}
1 & 0 & 0 & 0\\
0 & 1 & 0 & 0\\
0 & 1 & 1 & 0\\
0 & 0 & 0 & 1\\
\end{pmatrix},$$
the corresponding $\Lambda^2A$ are  $n_1n_2, s_1s_2,
p_{12}n_1(=n_2p_{12})$  and $\alpha_{12}$ respectively. By
symmetry, $n_in_j, s_is_j, p_{ij}n_i, \alpha_{ij},
s_i\alpha_{ij}s_i, s_j\alpha_{ij}s_j$ are also in the image of
$\Lambda^2$. Let $N$ be the subgroup of $A(3U)$ generated by these
elements. So $N\subset D(T^4)$.

Using the following relations ($i, j, k$ distinct)
$$n_i^2=s_i^2=p_{ij}^2=1,\ s_is_j=s_js_i,\ n_in_j=n_jn_i,\ p_{ik}p_{ij}=p_{jk}p_{ik},$$
$$n_is_t=s_tn_i,\ n_ip_{ij}=p_{ij}n_j,\ n_kp_{ij}=p_{ij}n_k,\ s_ip_{ij}=p_{ij}s_j,\ s_kp_{ij}=p_{ij}s_k,$$
$$n_i\alpha_{ij}=\alpha_{ij}^{-1}n_i,\ n_j\alpha_{ij}=\alpha_{ij}^{-1}n_j,\ n_k\alpha_{ij}=\alpha_{ij}n_k,$$
$$s_k\alpha_{ij}=\alpha_{ij}s_k,\ p_{ij}\alpha_{ij}p_{ij}=\alpha_{ji},\ p_{ik}\alpha_{ij}p_{ik}=\alpha_{kj},$$
we know that $N$ is a normal subgroup of $A(3U)$ of index 4.
Hence $[A(3U):D(T^4)]\leq 4$.

To finish our proof, we only need to show that the automorphisms $n_3,
s_3, n_3s_3$,  which give different cosets of $N$, are not in
$D(T^4)$. They have the same coefficients in $C(x_1), C(y_1),
C(x_2), C(y_2)$, which are
\begin{align}\label{cancel}
p_{kl,ij}=\delta_{ki}\delta_{lj}-\delta_{kj}\delta_{li},\ j\neq
\tau (i),\ \tau=(14)(23)\in S_4.
\end{align}
We want to use them and relations $P_{kl,ij}$ to give constraints
on $a_{ij}$ and show that $C(x_3), C(y_3)$ are also determined.
The linear combination $a_{sj}P_{kl,ij}-a_{kj}P_{sl,ij}$ gives new
relation
$$R_{k,l,s;i,j}: a_{lj}p_{ks,ij}-a_{sj}p_{kl,ij}+a_{kj}p_{sl,ij}=0.$$
\eqref{cancel} implies $a_{lj}=0$ if $l\neq j, \tau(j)$.
$P_{ij,ij}, j\neq \tau(i)$ becomes $a_{ii}a_{jj}=1$. So $a_{ii}=a_{11}=\pm 1$ for any $i$.
Now $P_{li,\tau(l)\tau(i)}$ implies $a_{i\tau(i)}=0$.
So $(a_{ij})=\pm I$. This shows that $C=I$ and $n_3, s_3, n_3s_3$ are not in the image of $\Lambda^2$.

\end{proof}

The following two remarks provide some Lie group insights in the K3 case and the $T^4$ case respectively. 

\begin{rem}\label{spinor norm}
If $\Gamma$ is indefinite, Aut$(\Gamma\otimes\R)$ has two components and the identity  component  Aut$^0(\Gamma\otimes\R)$ consists of   the automorphisms of spinor norm 1
(\cite{FM94} p.397). For the K3 surface, $D$ is exactly 
the intersection of $A(\Gamma)$ with Aut$^0(\Gamma\otimes\R)$.
\end{rem}

\begin{rem}\label{lattice} Here we explain that $D(T^4)$ has
finite index in $A(3U)$ via algebraic group theory, which was
communicated to us by S. Adams. We refer to the book of
Platonov-Rapinchuk \cite{PR} for relevant definitions and
theorems. Let $\Lambda^2: SL(4,\C)\to O(3U,\C)$ be the
$\mathbb{Q}$-morphism of algebraic groups defined as in the proof
of Proposition~\ref{kahler'}: $A=(a_{ij}) \mapsto
\Lambda^2A=(p_{kl,ij})$, where
$p_{kl,ij}=a_{ki}a_{lj}-a_{kj}a_{li}$. It is easy to see that the
kernel consists of $\pm$ id (Section 2 in \cite{LT}). As
$SL(4,\C)$ and $O(3U,\C)$ have same dimension, the image of
$\Lambda^2$ contains a neighborhood of the identity of $O(3U,\C)$.
The connectedness of $SL(4,\C)$ then implies that image of
$\Lambda^2$ is $O^0(3U,\C)$, the identity component of $O(3U,\C)$.
By Theorem 4.13 in \cite{PR}, p. 213 (or Theorem 4.14, p. 220),
$SL(4,\Z)$ is a lattice in $SL(4,\R)$ (i.e., a discrete subgroup
such that $SL(4,\R)/SL(4,\Z)$ has finite invariant volume).
Theorem 4.1 on page 204 of \cite{PR} then implies that
$\Lambda^2(SL(4,\Z))$ is an arithmetic subgroup of $O(3U,\C)$,
i.e., $\Lambda^2(SL(4,\Z))\cap A(3U) =\Lambda^2(SL(4,\Z))$ has
finite index in $A(3U)$.
\end{rem}

\subsection{Known non-K\"ahler CY surfaces}
The only known examples of non-K\"ahler CY are $T^2-$bundles over $T^2$.

If further,  $A(\Gamma)$ is infinite, then according to the table
above, $M$ is a $T^2-$bundles over $T^2$ with  $b^+=b^-=2$.
Such a manifold is a so-called Kodaira-Thurston manifold.

As a $T^2$ bundle over $T^2$, $M$ is
described by a triple
$$\{
\begin{pmatrix}
1 &\lambda \\
0 & 1
\end{pmatrix},
I, (0,0)\}, $$ where the first two matrices in $SL(2, \Z)$ are monodromies and the
third term denotes Euler numbers.

\begin{prop}\label{KT}
If $M$ is a $T^2-$bundle over $T^2$ with $b^+=2$, then $D(M)$ is  of infinite index.
\end{prop}
\begin{proof} It is convenient to  use
the following description of $M$. 
Let $G=Nil^3\times E^1$ act on $X=\mathbb{R}^4$ from left as
$$(x_0,y_0,z_0,t_0)(x,y,z,t)=(x_0+x,y_0+y,z_0+z+\lambda x_0y,t_0+t)$$
and $L$ is a discrete subgroup of $G$ generated by
$$(1,0,0,0), (0,1,0,0),(0,0,1,0),(0,0,0,1).$$
$M$ is then the quotient $L\backslash X$. 
$G-$invariant $1-$forms are generated by
$$dx,dy,dz-\lambda ydx,dt,$$
and 
$H^1=<dx, dy,dt>$.
$H^2$ is generated by
$$F_1=dx\wedge dt,\ F_2=dy\wedge (dz-\lambda ydx), \ F_3= dy\wedge
dt,\ F_4=dz\wedge dx$$  with $F_1\cdot F_2=F_3\cdot F_4=1$. Hence
$\Gamma=2U$. 

Observe that Im$(H^1\times H^1\stackrel{\wedge}{\longrightarrow}H^2)=<F_1, F_3>$ is 
$2-$dimensional isotropic subspace of $H^2$ invariant under $D(M)$.
  By
Lemma \ref{transitive'}, $D(M)$ has infinite index.

\end{proof}

We can also show that $D(M)$ is infinite. Let $\alpha=(x, y)$. For
any $T\in GL(2,\mathbb{Z})$, there exists a matrix $B\in M_{2\times 2}(\mathbb{Q})$ such that the map $\phi_T:\mathbb{R}^4\rightarrow \mathbb{R}^4$ defined as
$$\phi_T (\alpha, z,t)=(\alpha T, det(T)z+\alpha B\alpha^t, t)$$
preserves $L$.
So $D(M)$ is infinite for any $\lambda$.


\ms \ni{\it Acknowledgment}. The authors appreciate useful
discussions with Scot Adams. The research for the first named
author is partially supported by NSFC Grant 10990013. The research
for the second named author is partially supported by NCTS
postdoctoral fellowship. The research for the third named author
is partially supported by NSF.

\end{document}